\documentclass{svproc}

\usepackage{url}
\usepackage[utf8]{inputenc}
\usepackage{latexsym}
\usepackage{amssymb}
\usepackage{amsfonts}
\usepackage{amsmath}
\usepackage{bbm}
\usepackage{epsf}
\usepackage{epsfig}
\usepackage{graphicx}

\begin{document}

\mainmatter  

\title{Discrete-Time System of an Intracellular Delayed HIV Model with CTL Immune Response\thanks{This is
a preprint whose final form is published by Springer Nature Switzerland AG in the book 
'Dynamic Control and Optimization'.}}

\titlerunning{Discrete-Time System of an HIV Model}  

\author{Sandra Vaz\inst{1} \and Delfim F. M. Torres\inst{2}}

\authorrunning{S. Vaz and D. F. M. Torres}

\tocauthor{Sandra Vaz and Delfim F. M. Torres}

\institute{Center of Mathematics and Applications (CMA-UBI),\\ 
Department of Mathematics, University of Beira Interior, 
6201-001 Covilh\~{a}, Portugal\\
\email{svaz@ubi.pt},
\texttt{https://orcid.org/0000-0002-1507-2272}
\and
Center for Research and Development in Mathematics and Applications (CIDMA),\\ 
Department of Mathematics, University of Aveiro, 3810-193 Aveiro, Portugal\\ 
\email{delfim@ua.pt}, 
\texttt{https://orcid.org/0000-0001-8641-2505}}

\maketitle  

% ------------------------------------------------------

\begin{abstract}
In [Math. Comput. Sci. 12 (2018), no.~2, 111--127], a delayed model 
describing the dynamics of the Human Immunodeficiency Virus (HIV) 
with Cytotoxic T Lymphocytes (CTL) immune response is investigated
by Allali, Harroudi and Torres. Here, we propose a discrete-time 
version of that model, which includes four nonlinear difference 
equations describing the evolution of uninfected, infected, 
free HIV viruses, and CTL immune response cells and includes 
intracellular delay. Using suitable Lyapunov functions, 
we prove the global stability of the disease free equilibrium point 
and of the two endemic equilibrium points. We finalize by making some 
simulations and showing, numerically, the consistence 
of the obtained theoretical results. 


\end{abstract}

% ------------------------------------------------------

\section{Introduction}

Several mathematical models have been developed to better understand the dynamics 
of the HIV disease \cite{Ciu,Kirs,Nov,Per:99}. Human immunodeficiency virus (HIV)  
causes acquired immunodeficiency syndrome (AIDS), which is considered the end-stage 
of the infection. In this stage, the immune system fails to protect the whole body 
against harmful intruders. This happens because of the destruction of most of CD4+ T 
cells by the HIV virus, reducing them to fewer than 200 cells \cite{Gallo,Weiss}. 
Among available mathematical models, in \cite{MyID:318} HIV and tuberculosis 
coinfection is investigated. A particular case, 
using real data from Cape Verde islands, has been carried out in \cite{MyID:359},  
while the discrete case was analyzed in \cite{VD}, showing that ending AIDS epidemic 
by 2030 is a nontrivial task. Several models introduce the effect of cellular immune 
response, also called the cytotoxic T-lymphocyte (CTL) response, which attacks and kills 
the infected cells, see, for instance, \cite{Deboer,RSD,Sta}. The models show that 
this cellular immune response can control the load of HIV viruses. In \cite{Cul}, 
it is assumed that CTL proliferation depends, besides infected cells, as usual, 
also on healthy cells. Recently, the same problem was tackled by introducing time delays 
\cite{Elaiw,RSD}, which is justified by the fact that uninfected cells must be in contact 
with the HIV virus before they become infected. In \cite{HIV}, the investigation continued 
and the proposed basic model, illustrating this type of scenario, is
\begin{equation}
\label{eq:model}
\begin{cases}
\dot{X}(t) = \lambda -d X(t) - \beta X(t) V(t), \\[0.2 cm]
\dot{Y}(t) = \beta X(t) V(t) - a Y(t) - p Y(t) Z(t), \\[0.2 cm]
\dot{V}(t) =a N Y(t)- \mu V(t),\\[0.2 cm]
\dot{Z}(t) = cX(t)Y(t)Z(t)- s Z(t),
\end{cases}
\end{equation}
with given initial conditions $X(0) = X_{0}$, $Y(0) = Y_{0}$, $V(0) = V_{0}$, and $Z(0) = Z_{0}$. 
In this model \eqref{eq:model}, $X(t)$, $Y(t)$, $V(t)$ and $Z(t)$ denote, respectively, the concentrations 
at time $t$ of uninfected cells, infected cells, HIV virus, and CTL cells. The healthy CD4+ cells 
grow at a rate $\lambda$, decay at a rate $d X(t)$, and become infected by the virus at a rate 
$\beta X(t) V(t)$. Infected cells $Y$ die at a rate $a$ and are killed by the CTL response 
at a rate $p$. Free virus $V$ is produced by the infected cells at a rate a $N$ and decay 
at a rate $\mu$, where $N$ is the number of free virus produced by each actively infected 
cell during its life time. Finally, CTLs $Z$ expand in response to viral antigen derived 
from infected cells at a rate $c$ and decay in the absence of antigenic stimulation at a rate $s$. 
Our starting point here will be an extension of the continuous model \eqref{eq:model},
composed by nonlinear delayed ordinary differential equations. 
For most of these types of systems we cannot find the exact analytical solution. 
To perform numerical simulations using digital computers, we need to discretize the systems 
\cite{Elaydi:05}. There are several methods that allow us to discretize a model. One that 
has presented interesting results, and that we use here, is the nonstandard finite discrete 
difference (NSFD) scheme of Mickens \cite{Mickens:94,Mickens:02,Mickens:05,Mickens:07}.

Our work is organized as follows. Section~\ref{section2} is devoted to the delayed version 
of the continuous model \eqref{eq:model}, presenting its equilibrium points and available 
results about their stability. Section~\ref{section3} is dedicated to the presentation 
of our discrete model and the proof of existence, positivity and boundedness of solutions. 
We end the section by proving the global stability of the equilibrium points, using 
suitable Lyapunov functions, followed by some numerical simulations. Finally, 
conclusions are given in Section~\ref{section4}.

% ---------------------------------------------------------

\section{Preliminaries}
\label{section2}

We start by presenting the continuous-time model with delays
that serves as the basis of our current work,
as well as results regarding the stability of its equilibrium points.

In order to be realistic, in \cite{HIV} it has been introduced an intracellular time delay 
to the system of equations \eqref{eq:model}. Then, the model takes the following form:
\begin{equation}
\label{eq:modelD}
\begin{cases}
\dot{X}(t) = \lambda - d X(t) - \beta X(t) V(t),\\[0.2 cm]
\dot{Y}(t) = \beta X(t-\tau) V(t-\tau) - a Y(t) - p Y(t) Z(t), \\[0.2 cm]
\dot{V}(t) =a N Y(t)- \mu V(t),\\[0.2 cm]
\dot{Z}(t) = cX(t)Y(t)Z(t)-s Z(t).
\end{cases}
\end{equation}
Here, the delay $\tau$ represents the time needed for infected cells to produce virions 
after viral entry. The  model \eqref{eq:modelD} is a system of delayed ordinary differential 
equations. For such kind of problems, initial functions need to be addressed and an appropriate 
functional framework needs to be specified. Following \cite{HIV}, we consider
the Banach space $X = C([-\tau,0]; \mathbbm{R}^{4})$ of continuous mappings from
$[-\tau,0]$ to  $\mathbbm{R}^{4}$, equipped with the sup-norm  
$\|\phi \| = \underset{-\tau \leq t \leq 0}{\sup} |\phi(t)|$. 
It is assumed that the initial functions verify 
$(X(\theta), Y(\theta), V(\theta), Z(\theta)) \in X$. Also, from biological reasons, 
these initial functions $X(\theta)$, $Y(\theta)$, $V(\theta)$ and $Z(\theta)$ 
have to be nonnegative: $X(\theta)\geq 0$, $Y(\theta) \geq 0$, $V(\theta)\geq 0$, 
$Z(\theta)\geq 0$, for $\theta \in [-\tau, 0]$.  In Theorem~1 of \cite{HIV} it is  
proved that any solution of this system, satisfying certain conditions, is nonnegative 
and bounded for all $\tau \geq 0$. Moreover, the continuous model has three equilibrium points:
\begin{equation*} 
E_{0}=\left( \dfrac{\lambda}{d}, 0,0,0 \right)
\end{equation*}
and two endemic equilibrium points given by
\begin{equation*}
E^{\ast}=\left(\dfrac{\mu}{N \beta}, \dfrac{\lambda \beta N - d \mu}{a N \beta}, 
\dfrac{\lambda \beta N - d \mu}{\mu \beta}, 0\right)
\end{equation*}
and 
\begin{equation*}
\overline{E}=\left(\dfrac{\lambda \mu c - \beta a N s}{d \mu c}, 
\dfrac{d s \mu}{\lambda \mu c - \beta a N s}, 
\dfrac{d s a N}{ \lambda \mu c- \beta a N s}, 
\frac{\beta a N}{\mu p} \left( \dfrac{\lambda \mu c 
- \beta a N s}{d \mu c}\right)-\dfrac{a}{p} \right).
\end{equation*}
Regarding the stability of the disease-free equilibrium $E_{0}$, 
the following result was proved.

\begin{theorem}[See Theorem~2 of \cite{HIV}]
The local stability of the disease-free equilibrium point depends 
on the value $N \beta \lambda- d \mu$. Precisely, 
\begin{enumerate}
\item if $N \beta \lambda - d \mu< 0$, then the disease-free equilibrium 
point $E_{0}$ is locally asymptotically stable for any time delay $\tau \geq 0$;

\item if  $N \beta \lambda - d \mu > 0$, then the equilibrium $E_{0}$ 
is unstable for any time delay $\tau \geq 0$.
\end{enumerate}
\end{theorem}

For the local stability of the infected-equilibrium $E^{\ast}$, 
the following result holds.

\begin{theorem}[See Theorem~3 of \cite{HIV}]
The local stability of the disease-free equilibrium $E^{\ast}$ depends 
on the value of $\beta N (\mu c \lambda - \beta s a N) - \mu^{2} c d$. 
Precisely, 
\begin{enumerate}
\item if $\beta N (\mu c \lambda - \beta s a N) - \mu^{2} c d < 0$, 
then $E^{\ast}$ is locally asymptotically stable for any positive time delay;

\item if $\beta N (\mu c \lambda - \beta s a N) - \mu^{2} c d > 0$, 
then $E^{\ast}$ is unstable for any time delay. 
\end{enumerate}
\end{theorem}

For the second endemic equilibrium point $\overline{E}$, 
the following result has been proved.

\begin{theorem}[See Theorem~4 of \cite{HIV}]
Assume that  $\mu c \lambda - \beta s a N > 0$. 
If $\beta N (\mu c \lambda - \beta s a N) - \mu^{2} c d > 0$, 
then the infected equilibrium point $\overline{E}$ 
is locally asymptotically stable for $\tau=0$.
\end{theorem}

For $\tau >0$, the stability of $\overline{E}$ remains open.
Here we provide, for the first time in the literature,
a proper discrete-time version of the HIV model \eqref{eq:modelD}.

% -----------------------------------------------------------------

\section{Main Results}
\label{section3}

We begin this section by presenting our discrete-time model. 
Afterwards, we show the well-posedness of the model, that is, 
we prove that its solutions are positive and bounded. Moreover, 
we show that the equilibrium points are the same of the continuous model. 
We finalize this section by proving the global stability of each equilibrium point. 
For that we use suitable Lyapunov functions. We end this section by presenting 
some numerical simulations, which show consistence with the obtained theoretical results.

% ----------------------

\subsection{The discrete-time model}

One of the important features of the discrete-time epidemic models 
obtained by Mickens' method is that they present the same features 
as the corresponding original continuous-time models. Here, we construct 
a dynamically consistent numerical NSFD scheme for solving 
\eqref{eq:modelD} based on \cite{Mickens:94,Mickens:02,Mickens:05,Mickens:07}.  
Let us  define the time instants $t_{n}= nh$ with $n$ integer, the step size as
$h=t_{n+1} -t_{n}$, and $(X_{n}, V_{n}, Y_{n}, Z_{n})$ 
as the approximated values of $(X(nh), V(nh), Y(nh), Z(nh))$. 

Discretizing system \eqref{eq:modelD} using the NSFD scheme, we obtain:

\begin{equation}
\label{discmodel}
\begin{cases}
\dfrac{X_{n+1}-X_{n}}{\phi(h)} = \lambda - d X_{n+1} - \beta X_{n+1} V_{n},\\[0.2 cm]
\dfrac{Y_{n+1}-Y_{n}}{\phi(h)} = \beta X_{n-m+1} V_{n-m} - a Y_{n+1} - p Y_{n+1} Z_{n}, \\[0.2 cm]
\dfrac{V_{n+1}-V_{n}}{\phi(h)} = a N Y_{n+1} -  \mu V_{n+1},\\[0.2 cm]
\dfrac{Z_{n+1}-Z_{n}}{\phi(h)} =  c X_{n} Y_{n+1} Z_{n} - s Z_{n+1},
\end{cases}
\end{equation}
where the denominator function is $\phi(h)=h$ \cite{Elaiw,Mickens:07}. 
Throughout our study, for brevity, we write $\phi(h)=\phi$. 
Let us assume that there exists an integer $m \in \mathbbm{N}$ 
with $\tau=m \phi$. The initial conditions of system \eqref{discmodel} are
\begin{equation}
\label{condI}
X_{k}=\psi_{k}^{1}\geq 0; \quad
Y_{k}=\psi_{k}^{2}\geq 0; \quad 
V_{k}=\psi_{k}^{3}\geq 0; \quad
Z_{k}=\psi_{k}^{4}\geq 0
\end{equation}
for all $k=-m, -m+1, \ldots, 0$ and 
$\psi_{0}^{i}>0$, $i=1,2,3,4$.

Define the region 
\begin{equation*}
\Gamma = \left\{(x,y,v,z): 0 < X_{n}, Y_{n}, V_{n} \leq N_{1}, Z_{n} < N_{2}\right\},
\end{equation*}
where $N_{1}= \dfrac{a N \lambda}{Q}$,  
$Q=\min \{ d, \frac{a}{2}, \mu\}$ and 
$N_{2}=\dfrac{cX_{n}  \beta X_{n-m+1} V_{n-m}}{p s}$.

\begin{lemma}
Any solution $(X_{n},Y_{n},V_{n},Z_{n})$ of model \eqref{discmodel} 
with initial conditions \eqref{condI} is positive and ultimately bounded.
\end{lemma}

\begin{proof}
Since model \eqref{discmodel} is linear in $X_{n+1}$, $Y_{n+1}$, $V_{n+1}$, and $Z_{n+1}$, 
we can rewrite it as
\begin{equation}
\label{discmodelexp}
\begin{cases}
X_{n+1}=\dfrac{\lambda \phi + X_{n}}{ 1 + d \phi + \beta \phi V_{n}},\\[0.3 cm] 
Y_{n+1} = \dfrac{Y_{n}+\beta \phi  X_{n-m+1} Y_{n-m}}{1 + a \phi  + p \phi Z_{n}}, \\[0.3 cm]
V_{n+1}= \dfrac{V_{n} + a N \phi Y_{n+1}}{1 +  \mu \phi},\\[0.3 cm]
Z_{n+1} =  \dfrac{ Z_{n} +c \phi X_{n} Y_{n+1} Z_{n}}{1 + s \phi} .
\end{cases}
\end{equation}
Since all the parameters of model \eqref{discmodel} and the initial conditions 
are positive, it follows, by induction, that $X_{n} \geq 0$, $Y_{n} \geq 0$, 
$V_{n} \geq 0$, and $Z_{n} \geq 0$, for all $ n \in \mathbbm{N}$.
Regarding the boundedness of the solutions, let
\begin{equation}
\Omega_{n}=a N X_{n} + a N Y_{n+m} + \frac{a}{2} V_{n+m},
\end{equation}
from which
\begin{align*}
\Omega_{n+1}-\Omega_{n}
&= a N (X_{n+1}-X_{n}) 
+ a N (Y_{n+m+1}-Y_{n+m}) + \frac{a}{2} (V_{n+m+1}-V_{n+m})\\
&= a N \phi(\lambda - d X_{n+1} - \beta X_{n+1} V_{n})
+a N \phi (\beta X_{n+1} V_{n} - a Y_{n+m+1}\\ 
&\quad - p Y_{n+m+1} Z_{n+m})+\frac{a}{2} \phi( a N Y_{n+m+1} -\mu V_{n + m +1})\\
&= a N \lambda \phi - a N d \phi X_{n+1} 
- \frac{a^{2}}{2} \phi N Y_{n+m+1} - a N p \phi Y_{n+\tau +1} Z_{n+\tau}\\
&\quad -\frac{a}{2}\mu \phi V_{n+m+1}.
\end{align*}
Set $Q=\min \{ d, \frac{a}{2}, \mu\}$. Then, 
\begin{equation*}
\Omega_{n+1}-\Omega_{n} \leq a N \lambda \phi - Q \phi \Omega_{n+1}
\end{equation*}
so  that
$\Omega_{n+1}(1 + Q \phi) \leq a N \lambda \phi + \Omega_{n}$.
Hence, by \cite{Phi},
\begin{equation*}
\Omega_{n+1} \leq \dfrac{a N \lambda \phi}{1 + Q \phi} + \dfrac{\Omega_{n}}{1 + Q \phi}
\end{equation*}
and
\begin{equation}
\Omega_{n} \leq \left(  \dfrac{1}{1+ Q \phi}\right)^{n}\Omega_{0} 
+ \dfrac{a N \lambda}{Q} \left( 1- \left(\frac{1}{1+Q \phi}\right)^{n}\right) 
\end{equation}
so that $\underset{n \to \infty}{\limsup}\, \Omega_{n} \leq \dfrac{a N \lambda}{Q}=N_{1}$. 
Therefore, $\underset{n \to \infty}{\limsup} X_{n} \leq N_{1}$, 
$\underset{n \to \infty}{\limsup} Y_{n}\leq N_{1}$, 
and $\underset{n \to \infty}{\limsup} V_{n}\leq N_{1}$.
From the second and last equation of system \eqref{discmodel} we have
\begin{equation*}
Z_{n+1}-Z_{n} =  c \phi X_{n} Y_{n+1} Z_{n} - s \phi Z_{n+1}
\end{equation*}
or 
\begin{align*}
Z_{n+1}-Z_{n} &=  c \phi X_{n} \left(\beta X_{n-m+1} V_{n-m}
-a Y_{n+1}-\frac{(Y_{n+1}+Y_{n})}{\phi}\right) - s \phi Z_{n+1}\\
&\leq \frac{c \phi X_{n}\beta X_{n-m+1} V_{n-m}}{p} - s \phi Z_{n+1}.
\end{align*}
It follows from \cite{Phi} that
\begin{align*}
Z_{n+1} 
& \leq  \frac{ \frac{c \phi X_{n}\beta X_{n-m+1} V_{n-m}}{p}}{1+ s \phi} 
+ \frac{Z_{n} }{1+ s \phi}\\
&\leq \left( \dfrac{1}{1+s \phi}\right)^{n+1}Z_{0}
+\frac{c X_{n} \beta X_{n-m+1} V_{n-m}}{p s}\left( 1
- \left( \frac{1}{1+ s \phi} \right)^{n+1}\right).
\end{align*}
Consequently, 
$$
Z_{n} \leq \frac{c X_{n}(\beta X_{n-m+1} V_{n-m})}{p \ s}=N_{2}
$$ 
and every local solution $(X_{n}, Y_{n}, V_{n}, Z_{n})$ tends 
to $\Gamma$ as $n \to \infty$. \qed
\end{proof}

System \eqref{discmodel} has three equilibria:
\begin{enumerate}
\item[i)] the disease free equilibrium point 
$E_{0}=\left( \frac{\lambda}{d}, 0, 0, 0\right)$;

\item[ii)] the persistent infection equilibrium point without immune response,
\begin{equation*}
E^{\ast}=\left(  \frac{\mu}{\beta N}, \frac{\beta N \lambda -d \mu}{\beta N a}, 
\frac{\beta N \lambda - d \mu}{\beta \mu}, 0 \right);
\end{equation*}
\item[iii)] the persistent infection equilibrium with immune response, 
\end{enumerate}
\begin{equation*}
\overline{E}=\left( \frac{\lambda c \mu - \beta s a N}{d c \mu}, 
\frac{d \mu s}{ \lambda c \mu - \beta s a N}, 
\frac{s a N d}{ \lambda c \mu - \beta s a N}, 
\frac{\beta a N (\lambda c \mu - \beta s a N) - a d c \mu^{2}}{p d c \mu^{2}}\right).
\end{equation*}

The equilibrium point $E^{\ast}$ only exists if $\beta N \lambda -d \mu >0$, 
so let us define the basic reproduction number as
\begin{equation*}
\mathcal{R}_{0}:= \frac{\beta N \lambda}{d \mu}.
\end{equation*}
The equilibrium $\overline{E}$ only exists if 
$\beta a N (\lambda c \mu - \beta s a N) - a d c \mu^{2}>0$, 
so let us set the humoural immune response reproduction number as
\begin{equation*}
\mathcal{R}_{1}:=\dfrac{\beta N (\lambda c \mu - \beta s a N)}{d c \mu^{2}}.
\end{equation*}
Clearly,
\begin{equation*}
\mathcal{R}_{1}=\dfrac{\mathcal{R}_{0} (\lambda c \mu 
- \beta s a N)}{\lambda c \mu} < \mathcal{R}_{0}.
\end{equation*}
We can express the equilibrium points in terms of
$\mathcal{R}_{0}$ and $\mathcal{R}_{1}$ as follows:
\begin{align*}
E_{0}&=(X_{0},V_{0}, V_{0},Z_{0})=\left(\frac{\lambda}{d},0,0,0\right),\\
E^{\ast}&=(X^{\ast},Y^{\ast},V^{\ast}, Z^{\ast})=\left(\frac{X_{0}}{\mathcal{R}_{0}},
\frac{\lambda}{a \mathcal{R}_{0}} (\mathcal{R}_{0}-1),
\frac{N \lambda}{\mu \mathcal{R}_{0}}(\mathcal{R}_{0}-1),0\right),\\
\overline{E}&=(\overline{X},\overline{Y},\overline{V},\overline{Z})
=\left( \frac{\mathcal{R}_{1} \mu}{\beta N},  \frac{s \beta N}{\mu c \mathcal{R}_{1}}, 
\frac{\beta N^{2} a s}{\mu^{2} c \mathcal{R}_{1}},\frac{a (\mathcal{R}_{1}-1)}{p}\right).
\end{align*}
We can see from the previous relations that $E^{\ast}$ exists 
when $\mathcal{R}_{1}<1<\mathcal{R}_{0}$ and $\overline{E}$ only exists 
if  $\mathcal{R}_{0}>\mathcal{R}_{1}>1$.

% ----------------------

\subsection{Global stability}

In this section, we prove the  global stability of all the equilibria using 
suitable Lyapunov functions. We use the function $G(x)= x-\ln(x) -1$ 
that is positive for all $x > 0$ and $G(1)=0$. We make also use of the basic inequality
\begin{equation}
\label{ineqlya}
\ln(x) \le x-1.
\end{equation}

\begin{theorem}
Suppose that $\mathcal{R}_{0} \le 1$. 
Then $E_{0}$ of model \eqref{discmodel} 
is globally asymptotically stable.
\end{theorem}

\begin{proof}
Define the discrete Lyapunov function $L_{n}$ as
\begin{align*}
L_{n}(X_{n}, Y_{n}, V_{n},Z_{n})
&=\frac{1}{\phi} \left( 
X_{0} G\left(\frac{X_{n}}{X_{0}}\right) + Y_{n}+\frac{1+\mu \phi }{N}V_{n} 
+ \frac{p}{c N_{1}}(1+s \phi) Z_{n}\right)\\ 
&\qquad +\sum_{j=n-m}^{n-1} \beta X_{j+1} V_{j}.
\end{align*}
It follows that $L_{n}(X_{n}, Y_{n}, V_{n},Z_{n})>0$ for all 
$X_{n}\ge 0$, $Y_{n} \ge 0$, $V_{n} \ge 0$ and $Z_{n} \ge 0$. Moreover, 
$L_{n}(X_{n}, Y_{n}, V_{n},Z_{n})=0$ if $(X_{n}, Y_{n}, V_{n},Z_{n})=E_{0}$. 
Computing $\Delta L_{n}= L_{n+1}-L_{n}$, we have
\begin{align*}
&\Delta L_{n}=\frac{1}{\phi} \left( X_{0} G \left( 
\frac{X_{n+1}}{X_{0}}\right) + Y_{n+1}+\frac{1+\mu \phi }{N}V_{n+1} 
+ \frac{p}{c N_{1}}(1 + s \phi)Z_{n+1}\right) \\
&+\sum_{j=n-m+1}^{n} \beta X_{j+1} V_{j}\\
&- \left[\frac{1}{\phi} \left( X_{0} G \left( \frac{X_{n}}{X_{0}} \right) 
+ Y_{n}+\frac{1+\mu \phi }{N} V_{n} + \frac{p}{c N_{1}}(1+s\phi) Z_{n} \right)
+ \sum_{j=n-m}^{n-1} \beta X_{j+1} V_{j} \right]\\
&=\frac{1}{\phi}\left[X_{0} \left(  \frac{X_{n+1}}{X_{0}}
-\frac{X_{n}}{X_{0}}+ \ln \left( \frac{X_{n}}{X_{n+1}}\right)\right)
+(Y_{n+1}-Y_{n})+\frac{1+\mu \phi}{N}(V_{n+1}-V_{n})\right]\\
&+\frac{1}{\phi}\left[\frac{p}{cN_{1}}(1+s\phi)(Z_{n+1}-Z_{n})\right]
+\beta \left(\sum_{j=n-m+1}^{n}  X_{j+1} V_{j}
-\sum_{j=n-m}^{n-1} X_{j+1} V_{j} \right).
\end{align*}
Using \eqref{ineqlya}, we have
\begin{align*}
&\Delta L_{n} \leq \frac{1}{\phi}\left[\left( 1-\frac{X_{0}}{X_{n+1}}\right)\left( 
X_{n+1} -X_{n}\right)+(Y_{n+1}-Y_{n})+\frac{1+\mu \phi}{N}(V_{n+1}-V_{n})\right]\\
&\quad +\frac{1}{\phi}\left[\frac{p}{cN_{1}}(1+s\phi)(Z_{n+1}-Z_{n})\right]
+\beta \left( X_{n+1} V_{n}-  X_{n-m+1} V_{n-m} \right)\\
&= \left( 1-\frac{X_{0}}{X_{n+1}}\right)\frac{X_{n+1} 
-X_{n}}{\phi}+\frac{Y_{n+1}-Y_{n}}{\phi}+\frac{V_{n+1}-V_{n}}{N\phi}
+\frac{p(Z_{n+1}-Z_{n})}{cN_{1}\phi}\\
&+\beta \left(\sum_{j=n-m+1}^{n}  X_{j+1} V_{j}-\sum_{j=n-m}^{n-1} 
X_{j+1} V_{j} \right)+\frac{\mu}{N} (V_{n+1}-V_{n}) 
+ \frac{p s}{c N_{1}}(Z_{n+1}-Z_{n}).
\end{align*}
From the equations of system \eqref{discmodel},
\begin{align*}
\Delta L_{n}&\le \left( 1-\frac{X_{0}}{X_{n+1}}\right) \left(  
\lambda - d X_{n+1} - \beta X_{n+1} V_{n}\right)\\
&\quad +(\beta X_{n-m+1} V_{n-m} - a Y_{n+1} - p Y_{n+1} Z_{n})\\
&\quad+\frac{1}{N}(a N Y_{n+1} -  \mu V_{n+1})
+\frac{p}{cN_{1}}(c X_{n} Y_{n+1} Z_{n} - s Z_{n+1})\\
&\quad+\beta \left( X_{n+1} V_{n}-  X_{n-m+1} V_{n-m} \right)\\
&\quad+\frac{\mu}{N} (V_{n+1}-V_{n}) + \frac{s \ p}{c N_{1}}(Z_{n+1}-Z_{n})\\
&=\left(  \lambda - d X_{n+1}\right) \left( 1-\frac{X_{0}}{X_{n+1}}\right)
+\beta X_{0} V_{n} -p Y_{n+1} Z_{n} -\frac{\mu}{N} V_{n}\\ 
&\quad+\frac{p}{c N_{1}}c X_{n} Y_{n+1} Z_{n}-\frac{s \ p }{c N_{1}} Z_{n}.
\end{align*}
Using the first equation of \eqref{discmodel} at the equilibrium point $E_{0}$,
\begin{align*}
\Delta L_{n}& \le d \left( X_{0} - X_{n+1}\right) \left( 1-\frac{X_{0}}{X_{n+1}}\right)
+V_{n}\left(\frac{ \beta\lambda}{d} -\frac{\mu}{N}\right)+\left(\frac{p X_{n}}{ N_{1}} 
- p\right) Y_{n+1} Z_{n}\\
&\qquad - \frac{s \ p}{c N_{1}} Z_{n}\\
&= - \frac{d \left(X_{n+1} -X_{0} \right)^{2}}{X_{n+1}} 
-\frac{\mu}{N} V_{n} \left( 1-\mathcal{R}_{0}  \right)-p Y_{n+1} 
Z_{n}\left( 1-\frac{X_{n}}{N_{1}}\right)-\frac{s \ p}{c N_{1}}Z_{n}.
\end{align*}
Since $\mathcal{R}_{0} \le 1$ and $\underset{n \to \infty}{\lim \sup} X_{n}=N_{1}$, 
one has $\Delta L_{n} \le 0$ for all $n \ge 0$, that is, $L_{n}$ is a monotone decreasing 
sequence. If $L_{n} \ge 0$, then there is a limit for $\underset{n \to \infty}{\lim}  L_{n} \ge 0$. 
Therefore, $\underset{n \to \infty}{\lim} \Delta L_{n}= 0$ implies 
$\underset{n \to \infty}{\lim} X_{n}= X_{0}$ and 
$$
\underset{n \to \infty}{\lim}Y_{n}
= \underset{n \to \infty}{\lim}V_{n}=\underset{n \to \infty}{\lim}Z_{n}=0. 
$$
So, if $\mathcal{R}_{0} \le 1$, 
then $E_{0}$ is globally asymptotically stable. \qed
\end{proof}

\begin{lemma}
\label{demo2}
If $\mathcal{R}_{1}<1<\mathcal{R}_{0}$, then $Y^{\ast} < \overline{Y}$.
\end{lemma}

\begin{proof}
One can easily see that
\begin{align*}
Y^{\ast}-\overline{Y} 
&= \frac{\lambda (\mathcal{R}_{0} -1)}{a \mathcal{R}_{0}}
-\frac{\beta s N}{\mu c \mathcal{R}_{1}}
=\frac{\lambda \mu c\mathcal{R}_{1}(\mathcal{R}_{0}-1) 
-a s \beta N \mathcal{R}_{0}}{ a \mu c \mathcal{R}_{1} \mathcal{R}_{0}}\\
&< \frac{\lambda \mu c \mathcal{R}_{0} (\mathcal{R}_{0}-1) 
- a s \beta N \mathcal{R}_{0}}{a \mu c \mathcal{R}_{1} \mathcal{R}_{0}}
= \frac{\lambda \mu c (\mathcal{R}_{0}-1) - a s \beta N }{a \mu c \mathcal{R}_{1}}
\end{align*}
and $Y^{\ast}-\overline{Y}<0$, that is, 
$- \lambda \mu c < a s \beta N -\lambda \mu c  \mathcal{R}_{0} 
< a s \beta N -\lambda  \mu c$. The proof is complete. \qed 
\end{proof}

\begin{theorem}
If $\mathcal{R}_{1} \le 1 < \mathcal{R}_{0}$, 
then $E^{\ast}$ is globally asymptotically stable.
\end{theorem}

\begin{proof}
Define 
\begin{align*}
&\mathcal{L}_{n}(X_{n},Y_{n},V_{n}, Z_{n})\\
&=\frac{1}{\phi}\left[
X^{\ast} G\left( \frac{X_{n}}{X^{\ast}}\right)+ Y^{\ast} G\left( \frac{Y_{n}}{Y^{\ast}}\right)
+\frac{(1+\mu \phi) V^{\ast}}{N}G\left(\frac{V_{n}}{V^{\ast}}\right)
+ \frac{p(1+s \phi)}{c N_{1}}  Z_{n} \right]\\
&\quad +\beta X^{\ast} V^{\ast} \sum_{j=n-m}^{n-1} 
G\left( \frac{X_{j+1} V_{j}}{X^{\ast} V^{\ast}}\right).
\end{align*}
Then $\mathcal{L}_{n}(X_{n},Y_{n},V_{n},Z_{n})$ is positive 
for all $X_{n}, Y_{n}, V_{n}, Z_{n}$ strictly positive 
and it is equal to zero at $(X^{\ast},Y^{\ast},V^{\ast}, Z^{\ast})$. 
Computing $\Delta \mathcal{L}_{n}=\mathcal{L}_{n+1}-\mathcal{L}_{n}$, we get
\begin{align*}
\Delta \mathcal{L}_{n}
&=\frac{X^{\ast}}{\phi}G\left( \frac{X_{n+1}}{X^{\ast}}\right)
+ \frac{Y^{\ast}}{\phi} G\left( \frac{Y_{n+1}}{Y^{\ast}} \right)
+\frac{V^{\ast}(1+\mu \phi) }{N\phi}G\left( \frac{V_{n+1}}{V^{\ast}}\right)\\ 
&\quad + \frac{p(1+s \phi)}{c N_{1}\phi}Z_{n+1}
+\beta X^{\ast} V^{\ast}\sum_{j=n-m+1}^{n} 
G\left( \frac{X_{j+1} V_{j}}{X^{\ast} V^{\ast}}\right)\\
&=-\left[
\frac{X^{\ast}}{\phi}G\left( \frac{X_{n}}{X^{\ast}}\right)
+ \frac{Y^{\ast}}{\phi} G\left( \frac{Y_{n}}{Y^{\ast}}\right)
+\frac{ V^{\ast} (1+\mu \phi) }{N\phi}G\left( \frac{V_{n}}{V^{\ast}}\right)
+ \frac{p (1+s \phi) }{c N_{1}\phi}Z_{n} \right]\\
&\quad-\beta X^{\ast} V^{\ast}\sum_{j=n-m}^{n-1} G\left( \frac{X_{j+1} V_{j}}{X^{\ast} V^{\ast}}\right)\\
&=\frac{1}{\phi}\left( X^{\ast} \left( G \left( \frac{X_{n+1}}{X^{\ast} }\right) 
- G\left( \frac{X_{n}}{X^{\ast}}\right) \right) + Y^{\ast}\left( G\left( 
\frac{Y_{n+1}}{Y^{\ast}}\right) -G\left( \frac{Y_{n}}{Y^{\ast}}\right)\right) \right)\\
&\quad+\frac{1}{\phi}\left( \frac{p(1+s\phi)}{c N_{1}} \left( Z_{n+1}-Z_{n} \right) \right)
+\frac{V^{\ast} (1+\mu \phi)}{N}\left(G\left( \frac{V_{n+1}}{V^{\ast}}\right)
-G\left( \frac{V_{n}}{V^{\ast}}\right) \right)\\
&\quad+\beta X^{\ast} V^{\ast} \left( G\left( \frac{X_{n+1} V_{n}}{X^{\ast} V^{\ast}}\right)
-G\left( \frac{X_{n-m+1} V_{n-m}}{X^{\ast} V^{\ast}}\right)\right).
\end{align*}
Recalling inequality \eqref{ineqlya}, we have
\begin{equation*}
\begin{split}
G\left( \frac{\xi_{n+1}}{\xi^{\ast}}\right)- G\left( \frac{\xi_{n}}{\xi^{\ast}}\right)
&=\left( \frac{\xi_{n+1} -\xi_{n}}{\xi^{\ast}}\right)+\ln \left( \frac{\xi_{n}}{\xi_{n+1}}\right)\\ 
&\le \left( \xi_{n+1}-\xi_{n}\right)\left( \frac{1}{\xi^{\ast}}-\frac{1}{\xi_{n+1}}\right)
\end{split}
\end{equation*}
for $\xi = \{X, Y, V, Z\}$. Therefore, 
\begin{align*}
\Delta \mathcal{L}_{n}
&\le\frac{1}{\phi}\left( \left( X_{n+1}-X_{n}\right)\left( 1-\frac{X^{\ast}}{X_{n+1}}\right)
+  \left( Y_{n+1}-Y_{n}\right)\left( 1-\frac{Y^{\ast}}{Y_{n+1}}\right)\right)\\
&\quad +\frac{1}{\phi}\left(\frac{1}{N} \left( V_{n+1}-V_{n}\right)
\left( 1-\frac{V^{\ast}}{V_{n+1}}\right)+ \frac{p}{c N_{1}}  \left( Z_{n+1}-Z_{n}\right)\right)\\
&\quad+\frac{\mu V^{\ast}}{N}\left(G\left( \frac{V_{n+1}}{V^{\ast}}\right)
-G\left( \frac{V_{n}}{V^{\ast}}\right) \right)+\frac{p \ s}{c N_{1}}(Z_{n+1}-Z_{n})\\
&\quad+\beta X^{\ast} V^{\ast} \left( \frac{X_{n+1} V_{n}}{X^{\ast} V^{\ast}}
-\frac{X_{n-m+1} V_{n-m}}{X^{\ast} V^{\ast}} 
+\ln \left( \frac{X_{n-m+1} V_{n-m}}{X_{n+1} V_{n}}\right) \right).
\end{align*}
Using the equations of system \eqref{discmodel}, we have
\begin{align*}
\Delta \mathcal{L}_{n} 
&\le \left(\lambda - d X_{n+1} - \beta X_{n+1} V_{n}\right)\left( 1-\frac{X^{\ast}}{X_{n+1}}\right)\\
&\quad + \left(  \beta X_{n-m+1} V_{n-m} - a Y_{n+1} - p Y_{n+1} Z_{n}\right)\left( 1-\frac{Y^{\ast}}{Y_{n+1}}\right)\\
&\quad+\frac{1}{N} \left(  a N Y_{n+1} -  \mu V_{n+1}\right)\left( 1-\frac{V^{\ast}}{V_{n+1}}\right)
+ \frac{p}{c N_{1}}  \left(c X_{n}Y_{n+1}Z_{n}-s Z_{n+1}\right)\\
&\quad+\frac{\mu V^{\ast}}{N}\left( \frac{V_{n+1}}{V^{\ast}} 
- \frac{V_{n}}{V^{\ast}} +\ln\left( \frac{V_{n}}{V_{n+1}}\right) \right)
+\frac{p \ s}{c N_{1}}(Z_{n+1}-Z_{n})\\
&\quad+\beta X^{\ast} V^{\ast} \left( \frac{X_{n+1} V_{n}}{X^{\ast} V^{\ast}}
- \frac{X_{n-m+1} V_{n-m}}{X^{\ast} V^{\ast}} 
+\ln \left( \frac{X_{n-m+1} V_{n-m}}{X_{n+1} V_{n}}\right) \right).
\end{align*}
Expanding, simplifying, and using the conditions 
of system \eqref{discmodel} at $E^{\ast}$, where 
\begin{align*}
\lambda&=d X^{\ast} + \beta X^{\ast} V^{\ast},
\qquad \beta X^{\ast} V^{\ast}  = a Y^{\ast},
\qquad a N Y^{\ast}  = \mu V^{\ast}, 
\qquad  X^{\ast}=\frac{\mu}{\beta N},
\end{align*}
we get
\begin{align*}
\Delta \mathcal{L}_{n}
&\le \left( 1- \frac{X^{\ast}}{X_{n+1}}\right)(\lambda 
-d X_{n+1}) +\beta X^{\ast} V_{n}-p Y_{n+1} Z_{n} + a Y^{\ast}\\
&\quad -\beta X^{\ast} V^{\ast}\frac{X_{n-m+1}V_{n-m}Y^{\ast}}{x^{\ast} V^{\ast} Y_{n+1}}
+ p Y^{\ast} Z_{n}-aY^{\ast}\frac{V^{\ast} Y_{n+1}}{V_{n+1} Y^{\ast}} +\frac{\mu}{N}V^{\ast}\\
&\quad+\frac{p}{N_{1}}X_{n}Y_{n+1}Z_{n}-\frac{\mu}{N}V_{n}+\frac{\mu}{N}V^{\ast}
\ln\left(\frac{V_{n}}{V_{n+1}}\right)-\frac{p \ s}{c N_{1}} Z_{n}\\
&\quad+\beta X^{\ast} V^{\ast} \ln \left(\frac{X_{n-m+1} V_{n-m}}{X_{n+1} V_{n}} \right)\\ 
&\le -\frac{d (X_{n+1}-X^{\ast})^{2}}{X_{n+1}}-pY_{n+1}Z_{n} \left( 
1- \frac{X_{n}}{N_{1}}\right)+ p Z_{n} \left(Y^{\ast}- \overline{Y} \right)\\
&\quad+\beta X^{\ast} V^{\ast}\left[ 3- \frac{X^{\ast}}{X_{n+1}} 
- \frac{X_{n-m+1}V_{n-m} Y^{\ast}}{X^{\ast} V^{\ast} Y_{n+1}}
-\frac{V^{\ast} Y_{n+1}}{V_{n+1}Y^{\ast}}\right.\\
&\left.\qquad + \ln\left( 
\frac{X_{n-m+1} V_{n-m}}{X_{n+1} V_{n+1}}\right)\right]\\
&\le  -\frac{d (X_{n+1}-X^{\ast})^{2}}{X_{n+1}}-pY_{n+1}Z_{n} 
\left( 1- \frac{X_{n}}{N_{1}}\right)+ p Z_{n} \left(Y^{\ast}- \overline{Y} \right)\\
&\quad+\beta X^{\ast} V^{\ast}\left( -G\left(  \frac{X^{\ast}}{X_{n+1}} \right)
-G\left( \frac{X_{n-m+1}V_{n-m} Y^{\ast}}{X^{\ast} V^{\ast} Y_{n+1}}  \right)
-G\left(   \frac{V^{\ast} Y_{n+1}}{V_{n+1}Y^{\ast}}\right) \right).
\end{align*}
Hence, if  $\mathcal{R}_{1} \le 1 < \mathcal{R}_{0}$, and since 
$\underset{n \to \infty}{\lim \sup} X_{n}=N_{1}$ and Lemma~\ref{demo2} holds,
it follows that $\mathcal{L}_{n}$ is a monotone deceasing sequence. Since 
$\mathcal{L}_{n} \ge 0$, then  $\underset{n \to \infty}{\lim}{\mathcal{L}_{n}} \ge 0$. 
Therefore, $\underset{n \to \infty}{\lim} \Delta \mathcal{L}_{n}= 0$, 
which implies that $\underset{n \to \infty}{\lim} X_{n}= X^{\ast}$ 
and $\underset{n \to \infty}{\lim}Y_{n}=Y^{\ast}$, $\underset{n \to \infty}{\lim}V_{n}=V^{\ast}$, 
and $\underset{n \to \infty}{\lim}Z_{n}=Z^{\ast}$. Applying LaSalle's invariance principle, 
we conclude that $E^{\ast}$ is globally asymptotically stable.
\qed
\end{proof}

\begin{theorem}
Suppose that $\mathcal{R}_{1} >1$. 
Then $\overline{E}$ is globally asymptotically stable.
\end{theorem}

\begin{proof}
Define $\mathcal{U}_{n}(X_{n},Y_{n}, V_{n}, Z_{n})$ as
\begin{multline*}
\mathcal{U}_{n}(X_{n},Y_{n},V_{n}, Z_{n})
= \beta \overline{X} \cdot \overline{V} \sum_{j=n-m}^{n-1} G\left( 
\frac{X_{j+1} V_{j}}{\overline{X} \cdot \overline {V}}\right) 
+ p \overline{Y} \cdot \overline{Z} G\left( \frac{Z_{n}}{\overline{Z}} \right)\\
+\frac{1}{\phi}\left[
\overline{X} G\left( \frac{X_{n}}{\overline{X}}\right)
+ \overline{Y} G\left( \frac{Y_{n}}{\overline{Y}}\right)
+\frac{\beta \overline{X} \cdot \overline{V}(1
+\mu \phi) }{\mu}G\left( \frac{V_{n}}{\overline{V}}\right)
+ \frac{p \overline{Z}}{c \overline{X}} G\left(\frac{Z_{n}}{\overline{Z}} \right)\right].
\end{multline*}
Computing and simplifying $\Delta \mathcal{U}_{n}=\mathcal{U}_{n+1}-\mathcal{U}_{n}$, 
we have
\begin{align*}
\Delta \mathcal{U}_{n}
&= \frac{1}{\phi}\left[
\overline{X}\left( G\left( \frac{X_{n+1}}{\overline{X}}\right) 
-G\left( \frac{X_{n}}{\overline{X}}\right) \right)
+ \overline{Y} \left(G\left( \frac{Y_{n+1}}{\overline{Y}}\right)
- G\left( \frac{Y_{n}}{\overline{Y}}\right) \right)\right]\\
&+\frac{\beta \overline{X} \cdot \overline{V}}{\mu\phi} \left( 
G\left( \frac{V_{n+1}}{\overline{V}}\right)-G\left( \frac{V_{n}}{\overline{V}}\right)\right)
+ \frac{p \overline{Z}}{c\phi\overline{X}} \left(G\left(\frac{Z_{n+1}}{\overline{Z}} \right)
-G\left(\frac{Z_{n}}{\overline{Z}} \right)\right)\\
&+\beta \overline{X}. \overline{V}\left( \sum_{j=n-m+1}^{n} G\left( 
\frac{X_{j+1} V_{j}}{\overline{X}. \overline {V}}\right)
-\sum_{j=n-m}^{n-1} G\left( \frac{X_{j+1} V_{j}}{\overline{X}
\cdot \overline {V}}\right) \right)\\
&+ \beta  \overline{X} \cdot \overline{V} \left( G\left( \frac{V_{n+1}}{\overline{V}}\right)
-G\left( \frac{V_{n}}{\overline{V}}\right)\right) 
+ p \overline{Y} \cdot \overline{Z} \left( G\left( \frac{Z_{n+1}}{\overline{Z}} \right)
-G\left( \frac{Z_{n}}{\overline{Z}} \right) \right).
\end{align*}
It follows from inequality \eqref{ineqlya} that
\begin{equation*}
\overline{\xi} \left(G\left( \frac{\xi_{n+1}}{\overline{\xi}}\right)
- G\left( \frac{\xi_{n}}{\overline{\xi}}\right)\right) 
\le \left( 1-\frac{\overline{\xi}}{\xi_{n+1}}\right) 
\left( \xi_{n+1}-\xi_{n}\right)
\end{equation*}
for $\xi = \{X, Y, V, Z\}$. Therefore, 
$\Delta \mathcal{U}_{n}$ takes the form
\begin{align*}
\Delta \mathcal{U}_{n}
&= \frac{1}{\phi}\left[
\left( 1-\frac{\overline{X}}{X_{n+1}}\right) \left( X_{n+1}-X_{n}\right)
+\left( 1-\frac{\overline{Y}}{Y_{n+1}}\right) \left( Y_{n+1}-Y_{n}\right)\right]\\
&+\frac{1}{\phi}\left[\frac{\beta \overline{X}}{\mu} \left( 1
-\frac{\overline{V}}{V_{n+1}}\right) \left( V_{n+1}-V_{n}\right)
+ \frac{p}{c \overline{X}} \left( 1-\frac{\overline{Z}}{Z_{n+1}}\right)
\left( Z_{n+1}-z_{n}\right) \right]\\
&+\beta \overline{X}. \overline{V}\left( \sum_{j=n-m+1}^{n} G\left( 
\frac{X_{j+1} V_{j}}{\overline{X}. \overline {V}}\right)
-\sum_{j=n-m}^{n-1} G\left( \frac{X_{j+1} V_{j}}{\overline{X} 
\cdot \overline {V}}\right) \right)\\
&+ \beta  \overline{X}. \overline{V} \left( \frac{V_{n+1}}{\overline{V}}
- \frac{V_{n}}{\overline{V}} + \ln\left(\frac{V_{n}}{V_{n+1}} \right)\right) 
+p \overline{Y} \left( Z_{n+1} - Z_{n} 
+ \overline{Z} \ln \left(  \frac{Z_{n}}{Z_{n+1}}\right) \right).
\end{align*}
Now, expanding, simplifying, and using the conditions 
of system \eqref{discmodel} at $\overline{E}$, where 
\begin{align*}
\lambda&=d \overline{X} + \beta \overline{X} \overline{V},
\qquad \beta \overline{X} \cdot \overline{V}  
= \overline{Y}(a + p \overline{Z}), 
\qquad a N \overline{Y}  = \mu \overline{V},
\qquad  c \overline{X} \cdot \overline{Y}= s,
\end{align*}
we obtain
\begin{align*}
\Delta \mathcal{U}_{n}
&\le \left( 1- \frac{\overline{X}}{X_{n+1}}\right)(\lambda 
-d X_{n+1}-\beta X_{n+1} V_{n}) \\
&\quad +\frac{\beta \overline{X}}{\mu}\left( 
1- \frac{\overline{V}}{V_{n+1}}\right)\left( a N Y_{n+1} - \mu V_{n+1}\right)\\
&\quad+\left( 1- \frac{\overline{Y}}{Y_{n+1}}\right)\left( \beta X_{n-m+1} V_{n-m} 
- a Y_{n+1}- p Y_{n+1} Z_{n} \right)\\
&\quad+\frac{p}{c \overline{X}} \left( 
1- \frac{\overline{Z}}{Z_{n+1}}\right)\left( c X_{n}Y_{n+1}Z_{n}- s Z_{n+1}\right)\\ 
&\quad+\beta \overline{X}\cdot\overline{V} \left[ \frac{X_{n+1}V_{n}}{\overline{X}
	\cdot \overline{V}}-\frac{X_{n-m+1}V_{n-m}}{\overline{X}
	\cdot \overline{V}}+\ln\left(\frac{X_{n-m+1}V_{n-m}}{X_{n+1}V_{n}} \right)\right.\\
&\left. \qquad+ \frac{V_{n+1}}{\overline{V}}
-\frac{V_{n}}{\overline{V}}
+\ln\left( \frac{V_{n}}{V_{n+1}}\right)\right]
+ p \overline{Y} \left(Z_{n+1} - Z_{n} + \overline{Z} 
\ln\left( \frac{Z_{n}}{Z_{n+1}}\right) \right)\\
&\le p\overline{Y}\overline{Z}\left( -1-\frac{Y_{n+1}Z_{n}}{\overline{Y}\overline{Z}}
+ \frac{Y_{n+1}}{\overline{Y}}+\frac{X_{n}Y_{n+1}Z_{n}}{\overline{X} \overline{Y}\overline{Z}} 
-\ln \left( \frac{X_{n}Y_{n+1}}{\overline{X}\overline{Y}}\right)\right)\\
&\ +\beta \overline{X} \overline{V} \left( 3- \frac{\overline{X}}{X_{n+1}} 
- \frac{X_{n-m+1}V_{n-m} \overline{Y}}{\overline{X} \overline{V} Y_{n+1}}
-\frac{\overline{V} Y_{n+1}}{V_{n+1}\overline{Y}}
+ \ln\left( \frac{X_{n-m+1} V_{n-m}}{X_{n+1} V_{n+1}}\right)\right)\\
&\ -\frac{d (X_{n+1}-\overline{X})^{2}}{X_{n+1}}.
\end{align*}
Thus,
\begin{align*}
\Delta \mathcal{U}_{n} 
&\le - \beta \overline{X} \cdot 
\overline{Y} \left( G\left( \frac{X_{n-m+1}V_{n-m} \overline{Y}}{\overline{X} 
\overline{V} Y_{n+1}} \right) +G\left(\frac{\overline{X}}{X_{n+1}} \right)
+G\left(\frac{\overline{V} Y_{n+1}}{V_{n+1} \overline{Y}}\right)\right)\\
&\quad -p \overline{Y} \overline{Z} \left( -G\left(\frac{X_{n} Y_{n+1} Z_{n}}{\overline{X} 
\overline{Y} \overline{Z}}  \right) - G\left( \frac{Y_{n+1}}{\overline{Y}} \right) 
+G \left( \frac{Y_{n+1} Z_{n}}{\overline{Y}\overline{Z}} \right)\right)\\
&\quad -\frac{d}{X_{n+1}}(X_{n+1}-\overline{X})^{2}.
\end{align*}
Since $\mathcal{R}_{1}>1$, then $\overline{E}$ is strictly positive and 
$\Delta \mathcal{U}_{n}(X_{n},Y_{n},V_{n},Z_{n}) \le 0$ for all 
$n \ge m$, $m \in \mathbbm{N}$. It follows that $\mathcal{U}_{n}$ 
is a monotone decreasing sequence. We also have $\mathcal{U}_{n} \ge 0$. Then,  
$\underset{n \to \infty}{\lim}\mathcal{U}_{n} \ge 0$ and 
$\underset{n \to \infty}{\lim} \Delta \mathcal{U}_{n}= 0$, 
which implies that $\underset{n \to \infty}{\lim} X_{n}= \overline{X}$, 
$\underset{n \to \infty}{\lim}Y_{n}=\overline{Y}$, 
$\underset{n \to \infty}{\lim}V_{n}=\overline{V}$, 
and $\underset{n \to \infty}{\lim}Z_{n}=\overline{Z}$. 
Applying LaSalle's invariance principle, it follows that
$\overline{E}$ is globally asymptotically stable.
\qed
\end{proof}

% ----------------------

\subsection{Numerical simulations}

In this section, we perform some illustrative numerical simulations. 
In our simulations we use the values 
\begin{equation}
\begin{array}{c c c c c c c}
\lambda=1, & \ d=0.1, & \ p=0.0001, 
& \ s=0.2, & \ a=0.2, & \ \mu= 3, \ & N=750, 
\end{array}
\end{equation}
which satisfy the parameter ranges presented in Table~\ref{parameters},
and two sets of initial conditions: 
\begin{equation}
\begin{split}
\label{condI1}
\textrm{I}: & \ X_{k}=\psi_{k}^{1}=5, \ 
Y_{k}=\psi_{k}^{2}=1, \ V_{k}=\psi_{k}^{3}=1,
\  Z_{k}=\psi_{k}^{4}=2; \\
\textrm{II}: &\  X_{k}=\psi_{k}^{1}=15,\  
Y_{k}=\psi_{k}^{2}=2,\  
V_{k}=\psi_{k}^{3}=1,\  
Z_{k}=\psi_{k}^{4}=4;\\
\end{split}
\end{equation}
for all $k=-m, -m+1, \ldots, 0$.

% ----------------------------------------------
\begin{table}[t]
\caption{Parameters, symbols, meaning, and default values used in the HIV literature.}
\begin{center}
\begin{tabular}{l  l c l c l } \hline
Param. & Ref. &\hspace*{0.2 cm} & Meaning & \hspace*{0.2 cm}& Value  \\ \hline \hline
$\lambda$&\cite{Cul}& & source rate of CD4+ T cells& & $1-10$ cells $\mu l^{-1}$\\
&&& && $\textrm{days}^{-1}$ \\
$d$&\cite{Cul} & & Decay rate of healthy cells & & $0.007-0.1$ $\textrm{days}^{-1}$ \\
$\beta$&\cite{Cul}& & Rate at which CD4+T cells &&$0.00025-0.5 \mu l$ \\
&&&become infected &&$\textrm{virions}^{-1} \textrm{days}^{-1}$\\
$a$& \cite{Cul} && Death rate of infected CD4+T cells, && $0.2-0.3 \textrm{days}^{-1}$\\
&&&  not by CTL&&\\
$\mu$&\cite{Per}&& Clearance rate of virus && $2.06-3.81 \textrm{days}^{-1}$\\
$N$&\cite{Ciu}, \cite{Nov}&& Number of virions produced &&$6.25-23599.9 \mu l$\\
 &&&by infected CD4+T cells && $\textrm{virion} \ \textrm{days}^{-1}$\\
 $p$&\cite{Ciu},\cite{Paw}&& Clearance rate of infection &&$1-4.048 \times 10^{-4}$\\
  &&& && $\textrm{virion} \ \textrm{days}^{-1}$\\
 $c$&\cite{Ciu}&& Activation rate of CTL cells && $0.0051-3.912 \textrm{days}^{-1}$\\
 $h$&\cite{Ciu} && Death rate of CTL cells && $0.004-8.087 \textrm{days}^{-1}$ \\
 $\tau$& \cite{Bus}, \cite{Kahn} &&Time delay && $7-21 \textrm{days}^{-1}$ \\ \hline 
\end{tabular}
\end{center}
\label{parameters}
\end{table}
% ----------------------------------------------
 
For simulations regarding the stability of equilibria, 
we fix $\tau=2$ while $\beta$ and $c$ vary according
with cases I, II and III.

\bigskip
 
Case I. If $\beta=0.00025$ and $c=0.005$, 
then $\mathcal{R}_{0}=0.625<1$ and $\mathcal{R}_{1}=0.3125<1$. 
This means that $X_{n}$, the concentration of the uninfected cells, tends to 
$X_{0}=\frac{\lambda}{d}=10$ while $Y_{n}$, $V_{n}$ and $Z_{n}$ tend to zero.

\bigskip

Case II. If $\beta=0.0005$ and $c=0.01$, 
then $\mathcal{R}_{0}=1.25>1$ and $\mathcal{R}_{1}=0.625<1$. 
Therefore, the solutions of system \eqref{discmodel} tend 
to the equilibrium $E^{\ast}=\left( 8, 1, 50, 0 \right)$.  

\bigskip

Case III. If $\beta=0.0007$ and $c=0.1$, 
then $\mathcal{R}_{0}=1.75>1$ and $\mathcal{R}_{1}=1.6275>1$. 
This yields that all solutions of system \eqref{discmodel} 
tend to $\overline{E}=\left(  9.3, 0.215, 10.75, 1255 \right)$.  

\bigskip

% -------------------------------------------------
\begin{figure}[t]
\centering\includegraphics[width=8cm]{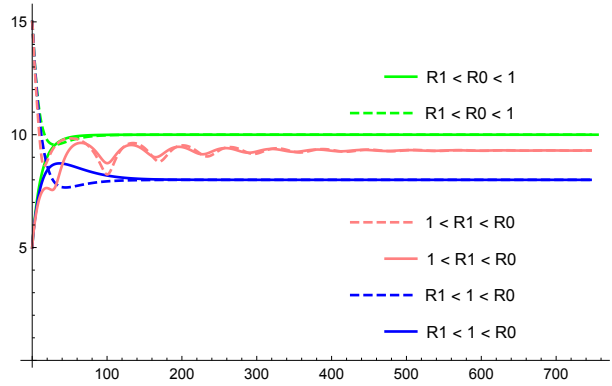}
\caption{The uninfected cells $X_{n}$ along time.}
\label{F1}
\end{figure}
% -------------------------------------------------
\begin{figure}[t]
\centering\includegraphics[width=8cm]{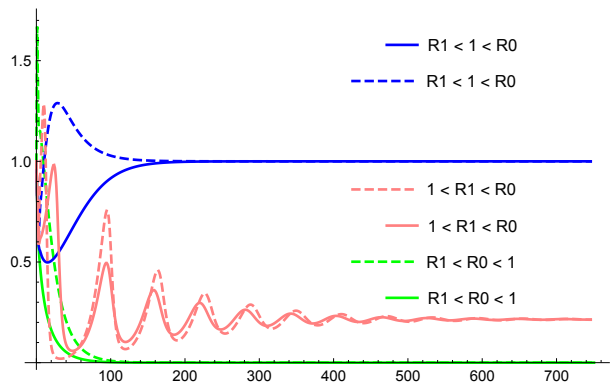}
\caption{The infected cells $Y_{n}$ along time.}
\label{F2}
\end{figure}
% -------------------------------------------------
\begin{figure}[!]
\centering\includegraphics[width=8cm]{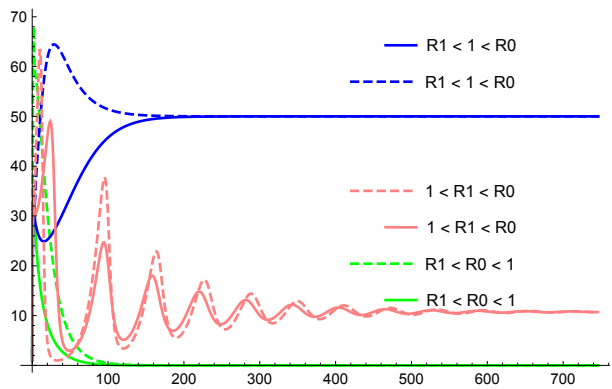}
\caption{The HIV virus $V_{n}$ along time.}
\label{F3}
\end{figure}
% -------------------------------------------------
\begin{figure}[!]
\centering\includegraphics[width=8cm]{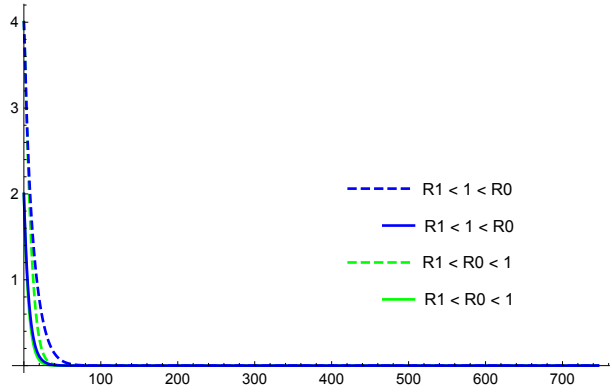}
\caption{The CTL cells $Z_{n}$ along time (Cases I and II).}
\label{F4}
\end{figure}
% -------------------------------------------------
\begin{figure}[!]
\centering\includegraphics[width=8cm]{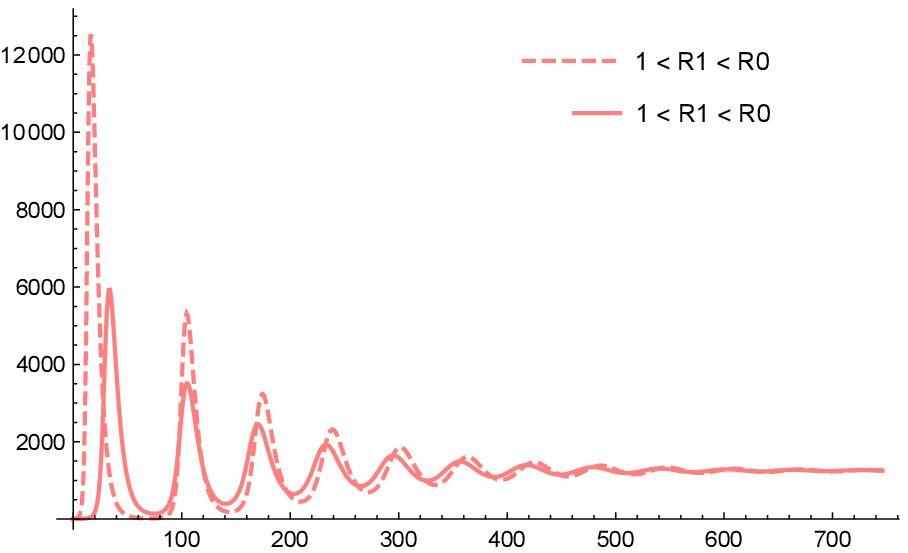}
\caption{The CTL cells $Z_{n}$ along time (Case III).}
\label{F5}
\end{figure}
% -------------------------------------------------

In Figures~\ref{F1}, \ref{F2} and \ref{F3}, it is represented the behavior 
$X_{n}$, $Y_{n}$ and $V_{n}$ for Cases I, II, and III. 
In Figure~\ref{F4}, it is represented Cases I and II, while 
in Figure~\ref{F5} it is represented the Case III, for better representation 
of the behavior $Z_{n}$ of the CTL cells.

 % ------------------------------------------------------

\section{Conclusion}
\label{section4}

In this work, we have proposed and studied the global stability of a delayed discrete-time 
HIV viral infection model with CTL immune response. The model describes the interaction 
between uninfected cells, infected cells, HIV free viruses, and CTL immune response, 
analogous to the continuous model investigated in \cite{HIV}. In the discrete case 
it was incorporated an intracellular time delay. For this model we prove the existence 
of positive and bounded solutions, showing that the model is well posed.  There are two 
threshold parameters, the basic reproduction number $\mathcal{R}_{0}$ and the immune 
response activation number $\mathcal{R}_{1}$. We determined the three equilibrium points 
and related their existence with the previous threshold numbers. Next, using suitable 
Lyapunov functions and LaSalle's invariance principle, we proved the global stability 
for each one of the equilibrium points, extending the results obtained in the continuous model. 
With the same data used in the literature for the continuous-time model, we made some 
simulations, which show the consistence between theoretical and numerical results.

% ------------------------------------------------------

\section*{Acknowledgements}

The authors were partially supported by 
the Portuguese Foundation for Science and Technology (FCT):
Sandra Vaz through the Center of Mathematics and Applications 
of \emph{Universidade da Beira Interior} (CMA-UBI), 
project UIDB/00212/2020; Delfim F. M. Torres through
the Center for Research and Development in Mathematics 
and Applications (CIDMA) of \emph{University of Aveiro}, 
project UIDB/04106/2020.

% ------------------------------------------------------

% ------------------------------------------------------

\end{document}